\documentclass[a4paper,11pt]{amsart}

\usepackage{url}
\usepackage{color}
\definecolor{darkgreen}{rgb}{0,0.5,0}
\definecolor{red}{rgb}{0.9,0,0}
\usepackage[
        colorlinks, citecolor=darkgreen,
        backref,
        pdfauthor={J. Steffen Mueller, Michael Stoll},
        pdftitle={Computing Canonical Heights in Quasi-Linear Time}
]{hyperref}

\usepackage{enumerate}
\usepackage{amssymb,amsmath, amsthm}
\usepackage[OT2,OT1]{fontenc}
\usepackage{fancyhdr}
\usepackage{colonequals}

\usepackage[alphabetic,backrefs]{amsrefs} 

\numberwithin{equation}{section}

\newtheorem{thm}{Theorem}[section]

\newtheorem{prop}[thm]{Proposition}
\newtheorem{lemma}[thm]{Lemma}

\theoremstyle{definition}

\newtheorem{algo}[thm]{Algorithm}

\theoremstyle{remark}

\newtheorem{rk}[thm]{Remark}
\newtheorem{ex}[thm]{Example}

\newcommand\Q{\mathbb{Q}}

\newcommand\A{\mathbb{A}}
\newcommand\Z{\mathbb{Z}}
\newcommand\R{\mathbb{R}}

\renewcommand\O{\mathcal{O}}

\newcommand{\To}{\longrightarrow}
\newcommand{\BP}{{\mathbb P}}
\newcommand{\eps}{\varepsilon}

\newcommand{\vmin}{v_{\textrm{min}}}

\newcommand{\surj}{\twoheadrightarrow}

\newcommand{\Mult}{\operatorname{\sf M}}
\newcommand{\lstrut}{{\Large\strut}}

\begin{document}

\title[Computing Canonical Heights in Quasi-Linear Time]{Computing Canonical Heights on
Elliptic Curves in Quasi-Linear Time}

\author{J. Steffen M\"uller}
\address{Institut f\"ur Mathematik,
          Carl von Ossietzky Universit\"at Oldenburg,
          26111 Oldenburg, Germany}
\email{jan.steffen.mueller@uni-oldenburg.de }

\author{Michael Stoll}
\address{Mathematisches Institut,
          Universit\"at Bayreuth,
          95440 Bayreuth, Germany.}
\email{Michael.Stoll@uni-bayreuth.de}

\date{December 15, 2015}


\begin{abstract}
We introduce an algorithm that can be used to compute the canonical height of a point on an elliptic
curve over the rationals in quasi-linear time.
As in most previous algorithms, we decompose the difference between the canonical and the
naive height into an archimedean and a non-archimedean term.
Our main contribution is an algorithm for the computation of the non-archimedean term
that requires no integer factorization and runs in quasi-linear time.
\end{abstract}

\maketitle


\section{Introduction}\label{intro}

Let $E$ denote an elliptic curve defined over a number field~$K$.
The canonical height is a quadratic form $\hat{h} \colon E(K)\otimes \R \to \R$, first
constructed by N\'eron~\cite{Neron} and Tate (unpublished).
For several applications, such as computing generators for~$E(K)$ and computing the
regulator appearing in the conjecture of Birch and Swinnerton-Dyer, one needs to
compute~$\hat{h}(P)$ for points $P \in E(K)$.

To this end, one typically chooses a Weierstrass equation~$W$ for~$E$ over~$K$ with $\O_K$-integral coefficients and
decomposes $\hat{h}(P)$ (or $\hat{h}(P) - h(P)$, where $h$ is the naive height on~$E$ with
respect to~$W$) into a sum of local terms, one for each place of~$K$.
For simplicity, let us assume $K=\Q$.
There are several efficient algorithms for the computation of the contribution at infinity,
see Section~\ref{arch}.
A very simple and efficient algorithm of Silverman~\cite{SilvermanHeights} can be used to compute the
non-archimedean contributions separately. However, in order to determine the non-archimedean places which contribute to
$\hat{h}(P)$ (or $\hat{h}(P) - h(P)$), the algorithm of~\cite{SilvermanHeights} assumes that the prime
factorization of the discriminant~$\Delta(W)$ is known, which renders this approach inefficient when the
coefficients of~$W$ are large.
This observation motivated Silverman's article~\cite{SilvermanLittleFact}, where it is
shown how to compute~$\hat{h}(P)$ without the need to factor~$\Delta(W)$.
Nevertheless, the algorithm of~\cite{SilvermanLittleFact} requires the prime factorization
of $\gcd(c_4(W), c_6(W))$ in order to find a globally minimal Weierstrass equation for~$E$.

In this note, we introduce an algorithm for the computation of~$\hat{h}(P)$ that does not
require any factorization into primes at all and runs in time quasi-linear in the size of
the input data and the desired precision of the result.
More precisely, let $\|W\|$ denote the largest absolute value of the coefficients of~$W$,
and let $d$ denote the number of desired bits of precision after the binary point.
We denote the time needed to multiply two $d$-bit integers by~$\Mult(d)$.
The following result is our main theorem.
Recall the `soft-O' notation: $f(n) \in \tilde{O}(g(n))$ means that
there are constants $c, m > 0$ such that for $n$ sufficiently large, $|f(n)| \le c g(n) (\log g(n))^m$.
Using fast multiplication algorithms, we have $\Mult(d) \in \tilde{O}(d)$.
We also use the notation $f(n) \ll g(n)$ to express the fact that there is a constant $c > 0$
such that $|f(n)| \le c g(n)$ for $n$ sufficiently large.

\begin{thm}\label{T:main}
  Let $E$ be given by a Weierstrass equation~$W$ with coefficients in~$\Z$
  and let $P \in E(\Q)$.
  Then we can compute~$\hat{h}(P)$ to $d$~bits of (absolute) precision in time
    \begin{align*}
      &\ll \log(d + h(P)) \Mult(d + h(P)) \\
      & \qquad{} + (\log\log \|W\|)
                   \Mult\bigl((\log\log \|W\|) (\log \|W\|)\bigr) \\
      & \qquad{} + \log(d + \log\|W\|)^2 \Mult(d + \log\|W\|) \\
      &{} \in \tilde{O}(d + h(P) + \log\|W\|) \, .
    \end{align*}
\end{thm}

Since the size of the input is measured by $h(P) + \log\|W\|$ ---
the first term gives the size of~$P$, the second term gives the size of~$W$ ---
and the size of the output is measured by $\log h(P) + d$, this means that we can
compute $\hat{h}(P)$ in quasi-linear time.

The strategy of the proof is to first find an algorithm for the computation of the local
non-archimedean contributions that does not assume minimality; see Proposition~\ref{P:fast algo}.
Building on this, the non-archimedean contribution to $\hat{h}(P) - h(P)$ can be computed upon
observing that it is a sum of rational multiples of logarithms of prime numbers, which can
be determined by working globally modulo a suitable power of~$\Delta(W)$.
Combining this with a complexity analysis of the fastest known algorithm for the
computation of the local height at infinity due to Bost-Mestre~\cite{BMisog},
Theorem~\ref{T:main} follows.
We note that Marco Caselli, working on his PhD under the supervision of John Cremona,
is currently extending the Bost-Mestre algorithm to also deal with complex places.

The paper is organized as follows.
In Section~\ref{formulas} we set up some notation and introduce the notion of Kummer
coordinates of points on elliptic curves.
Heights and their local decompositions are recalled in Section~\ref{heights}.
In Section~\ref{nonarch} we discuss an algorithm that allows us to compute a non-archimedean
local summand of $\hat{h}(P)-h(P)$ efficiently without assuming minimality, and we estimate
its running time.
Section~\ref{arch} contains a discussion of the algorithm of Bost-Mestre for the computation
of the local height at infinity and of its running time.
We then combine the non-archimedean and the archimedean results into an efficient
algorithm for the computation of $\hat{h}(P)$ in Section~\ref{algo}, leading to a proof
of Theorem~\ref{T:main}.
In the final Section~\ref{examples} we discuss the practicality of our algorithm.

\subsection*{Acknowledgments}

We would like to thank Jean-Fran\c{c}ois Mestre for providing us with a copy of the unpublished
manuscript~\cite{BMisog}, Mark Watkins for answering our questions about the
computation of canonical heights in {\sf Magma} and Elliot Wells for pointing out
an inaccuracy in Algorithm~\ref{algo2}.


\section{Kummer coordinates}\label{formulas}

Let $K$ be a field and consider an elliptic curve~$E/K$, given by a Weierstrass
equation
\begin{equation}\label{W_eqn}
W \colon y^2 + a_1xy + a_3y = x^3 + a_2x^2 + a_4x + a_6\,,
\end{equation}
where $a_1, a_2, a_3, a_4, a_6 \in K$.
As usual, let
\begin{align*}
  b_2 & = a_1^2+4a_2\,,\\
  b_4 & = 2a_4+a_1a_3\,,\\
  b_6 &= a_3^2+4a_6\,,\\
  b_8 &= a_1^2a_6+4a_2a_6-a_1a_3a_4+a_2a_3^2-a_4^2\,,
\end{align*}
and let
\[
  \Delta(W) = -b_2^2 b_8 - 8b_3^4 - 27b_6^2 + 9b_2 b_4 b_6
\]
denote the discriminant of the equation $W$.
Consider the functions $f$ and~$g$, defined for $P \in E(K)\setminus\{O\}$ by
\begin{align*}
 f(P)&= 4x(P)^3 + b_2x(P)^2 + 2b_4x(P) + b_6\,,\\
 g(P)&= x(P)^4 - b_4x(P)^2 - 2b_6x(P) - b_8\,.
\end{align*}
Then for $P \in E(K)\setminus E[2]$, we have $x(2P)=g(P)/f(P)$.
We now extend this to all $P \in E(K)$.

Note that $\BP^1$ is the Kummer variety $E/\{\pm 1\}$ of $E$.
An explicit covering map $E \to \BP^1$ is given by
\[\begin{array}{crcl}
\kappa \colon &E& \To& \BP^1\\
& (x:y:1)&\longmapsto& (x:1)\\
& O&\longmapsto& (1:0)\,.\\
\end{array}\]
We call $(x_1,x_2) \in \A^2(K)\setminus\{(0,0)\}$ a pair of {\em Kummer coordinates}
for $P \in E(K)$, if $\kappa(P) = (x_1:x_2)$.

The degree~4 homogenizations of $g$ and~$f$ are
\begin{align*}
 \delta_1(x_1,x_2)  &= x_1^4 - b_4x_1^2x_2^2 - 2b_6x_1x_2^3 - b_8x_2^4\, ,\\
 \delta_2(x_1,x_2)  &= 4x_1^3x_2 + b_2x_1^2x_2^2 + 2b_4x_1x_2^3 + b_6x_2^4\, ,
\end{align*}
respectively.
For $(x_1,x_2) \in \A^2_K$, we set
\[
  \delta(x_1,x_2) = (\delta_1(x_1,x_2), \delta_2(x_1,x_2))\,.
\]
It follows that if $(x_1,x_2)$ is a pair of Kummer coordinates for $P\in E(K)$, then
$\delta(x_1,x_2)$ is a pair of Kummer coordinates for $2P$.


\section{Heights}\label{heights}

Let $K$ be a number field and let $E/K$ be an elliptic curve, given by a Weierstrass
equation~$W$ as in~\eqref{W_eqn}.
We denote by~$M_K$ the set of places of~$K$. For a place $v \in M_K$, we
normalize the associated absolute value $|{\cdot}|_v$ so that it restricts to the
usual absolute value on~$\Q$ when $v$ is an infinite place and so that $|p|_v = p^{-1}$
when $v$ is a finite place above~$p$. We write $n_v = [K_v : \Q_w]$ for the local degree,
where $w$ is the place of~$\Q$ below~$v$. Then we have the product formula
$\prod_{v \in M_K} |x|_v^{n_v} = 1$ for all $x \in K^\times$.
The {\em naive height} of $P \in E(K)\setminus\{O\}$ with respect to~$W$ is given by
\[
  h(P) = \frac{1}{[K:\Q]}\sum_{v \in M_K}n_v \log\max\{|x_1|_v,|x_2|_v\}\,,
\]
where $(x_1, x_2)$ is a pair of Kummer coordinates for~$P$.
Note that $h(P)$ does not depend on the choice of $(x_1, x_2)$, by the product formula.

The limit
\[ \hat{h}(P) = \lim_{n \to \infty} \frac{h(nP)}{n^2} \]
exists and is called the {\em canonical height} (or {\em N\'eron-Tate height}) of~$P$.

For the computation of~$\hat{h}(P)$, the limit construction is not suitable due to slow
convergence and exponential growth of the size of the coordinates.
Instead, one decomposes~$\hat{h}(P)$ into local terms.
We now recall how this can be achieved, following~\cite{CPS}.
For $v \in M_K$ and $Q \in E(K_v)$, we set
\[
  \Phi_v(Q) = \frac{\max\{|\delta_1(x_1,x_2)|_v, |\delta_2(x_1,x_2)|_v\}}{\max\{|x_1|_v,|x_2|_v\}^4}
\]
where $(x_1, x_2) \in \A^2(K_v) \setminus \{0,0\}$ is a pair of Kummer coordinates
for~$Q$. Since $\delta_1$ and~$\delta_2$ are homogeneous of degree~$4$, $\Phi_v(Q)$
does not depend on the choice of~$(x_1, x_2$).
The function $\Phi_v$ is continuous and bounded on~$E(K_v)$, so it makes sense to define
\[
  \Psi_v(Q) = -\sum^{\infty}_{n=0} 4^{-n-1} \log\Phi_v(2^nQ)\, ,
\]
which is likewise continuous and bounded.
Note that for $P \in E(K)$, we have
\[
  h(2P) - 4h(P) = \frac{1}{[K:\Q]}\sum_{v \in M_K}n_v\log\Phi_v(P)\,,
\]
and Tate's telescoping trick yields the formula
\begin{equation}\label{can_height_formula}
  \hat{h}(P) = h(P) - \frac{1}{[K:\Q]}\sum_{v \in M_K}n_v\Psi_v(P)\,,
\end{equation}
which we will use to compute the canonical height.

It is also possible to decompose the canonical height into a sum of local height functions.
For $v \in M_K$ and $Q \in E(K_v)\setminus\{O\}$, we define the {\em local height} of~$Q$ as
\[
  \hat{\lambda}_v(Q) = \log \max\{1,|x(Q)|_v\} - \Psi_v(Q)\,.
\]
Then~\eqref{can_height_formula} immediately implies
\begin{equation}\label{loc_height_decomp}
  \hat{h}(P) = \frac{1}{[K:\Q]}\sum_{v \in M_K}n_v\hat{\lambda}_v(P)
\end{equation}
for $P \in E(K)\setminus\{O\}$.

\begin{rk}
  Several normalizations for the local height on elliptic curves can be found in the
  literature, see the discussion in~\cite{CPS}.
  Our normalization corresponds to that used in~\cite{CPS}, so in particular, our canonical
  height is twice the canonical height in Silverman's paper~\cite{SilvermanHeights} and
  in his books on elliptic curves.
  More precisely, we have
\[
  \hat{\lambda}_v(Q)=2\hat{\lambda}^{\text{SilB}}_v(Q)+\frac{1}{6}\log|\Delta(W)|_v\,,
\]
where $\hat{\lambda}^{\text{SilB}}_v$ is the normalization used in Silverman's
book~\cite{ATAEC}*{Chapter VI}. The advantages of our normalizations are discussed
in~\cite{CPS}; the crucial advantage of $\hat{\lambda}^{\text{SilB}}_v$
is its independence of the chosen Weierstrass equation.
\end{rk}

In Section~\ref{arch}, we need to know how local heights change under isogenies.

\begin{prop}\label{P:bernardi}(Bernardi~\cite{Bernardi})
Let $E$ and~$E'$ be elliptic curves defined over~$K_v$, given by respective Weierstrass equations
$W$ and~$W'$. Let $\varphi \colon E\To E'$
be an isogeny of degree~$n$. If $Q\in E(K_v)$ satisfies $\varphi(Q)\ne 0$, then we have
\[
 \hat{\lambda}_v(\varphi(Q))=n\hat{\lambda}_v(Q)-\log |F_\varphi(Q)|_v-\frac{1}{6}\log
 |m(\varphi)|_v,
\]
where
\[F_\varphi(Q)=\prod_{R\in \ker(\varphi)\setminus \{O\}}(x(Q)-x(R))\]
and \[m(\varphi)=\lim_{Q\to
O}\left(\frac{x(Q)}{x(\varphi(Q))}\right)^6\frac{\Delta(W')}{\Delta(W)}.\]
\end{prop}


\section{Non-archimedean local error functions}\label{nonarch}

In this section, we let $K$ be a non-archimedean local field with normalized additive valuation
$v \colon K\surj \Z\cup\{\infty\} $.
Let $\O$ denote the valuation ring of~$K$, let $k$ denote the residue class field of~$\O$
and let $\pi$ be a uniformizing element of~$\O$.
Consider an elliptic curve~$E/K$, given by a Weierstrass equation~$W$ as
in~\eqref{W_eqn}, with coefficients in~$\O$.

Given $P \in E(K)$, we choose a pair of Kummer coordinates $(x_1,x_2)$ for $P$ and define
\[
  \eps(x_1,x_2)=\min\{v(\delta_1(x_1,x_2)),v(\delta_2(x_1,x_2))\}-4\min\{v(x_1),v(x_2)\} \in
  \Z\,.
\]
Note that $\eps$ does not depend on the choice of Kummer coordinates, so we can define
$\eps(P) = \eps(x_1,x_2)$ for any such choice.
The function~$\eps$ is nonnegative, bounded and continuous in the $v$-adic topology.
Hence we can define
\begin{equation}\label{mu_def}
      \mu(P) = \sum^{\infty}_{n=0} \frac{1}{4^{n+1}} \eps(2^nP) \in \R\,.
\end{equation}
It follows that $\mu$ is nonnegative, bounded and continuous as well.
One can show that in fact $\mu(P) \in \Q$, compare Table~\ref{T:mu_values}.

\begin{rk}\label{R:completion}
  If $K$ is the completion of a number field at a non-archimedean place~$v$, then
  we have $n_v \log\Phi_v(P) = -\eps(P)(\log \#k)$
  and $n_v \Psi_v(P) = \mu(P) (\log \#k)$ for $P \in E(K)$,
  where $\Phi_v$ and~$\Psi_v$ are as defined in Section~\ref{heights}.
\end{rk}

If we have bounds for~$\eps(P)$ and for the denominator of~$\mu(P)$, then we can
use~\eqref{mu_def} to compute~$\mu(P)$.

\begin{lemma} \label{L:fast algo 2}
  Assume that $M \ge 2$ and~$B$ are nonnegative integers such that
  \begin{enumerate}[\upshape(1)]\addtolength{\itemsep}{1mm}
    \item $M'\mu(P) \in \Z$ for some $0<M'\le M$, and
    \item $\max \{\eps(P) : P \in E(K)\} \le B$.
  \end{enumerate}
  Set $\displaystyle m = \left\lfloor\frac{\log(BM^2/3)}{\log 4}\right\rfloor$.

  Then $\mu(P)$ is the unique fraction with denominator~$\le M$
  in the interval $[\mu_0, \mu_0 + 1/M^2]$, where
  \[ \mu_0 = \sum_{n=0}^{m} 4^{-n-1} \eps(2^n P) . \]
\end{lemma}

\begin{proof}
  We know that $\mu(P)$ is a fraction with denominator bounded by~$M$. Two distinct such
  fractions have distance greater than~$1/M^2$ (here we use $M \ge 2$),
  so there is at most one such fraction in the given interval.
  On the other hand, we know that
  \[ \mu_0 \le \mu(P) \le \mu_0 + \sum_{n>m} 4^{-n-1} B
                      = \mu_0 + \frac{B}{3\cdot 4^{m+1}}
                      \le \mu_0 + \frac{1}{M^2}\,. \qedhere \]
\end{proof}

We now discuss how to bound~$\eps(P)$ and the denominator of~$\mu(P)$.

\begin{lemma}\label{L:mu_eps_bounds}
  For $P \in E(K)$, we have
  \begin{enumerate}[\upshape (i)]\addtolength{\itemsep}{1mm}
\item
  $ 0 \le \mu(P) \le \frac{1}{4} v(\Delta(W))\,;$
\item
  $ 0 \le \eps(P) \le v(\Delta(W))\,.  $
\item
  The denominator of $\mu(P)$ is bounded from above by $v(\Delta(W))$.
  \end{enumerate}
\end{lemma}

\begin{proof}
  If the Weierstrass equation $W$ is minimal,
  then $\eps(P)$ (or, equivalently, $\mu(P)$) vanishes if and only if $P$ has nonsingular
  reduction, and $\eps(P)$ (or, equivalently, $\mu(P)$) depends only on the component
  of the special fiber of the N\'eron model of~$E$ that $P$ reduces to, see~\cite{SilvermanHeights}.
  For minimal~$W$, the nonzero values that $\mu$ can take and an upper bound~$\alpha$ are
  given in Table~\ref{T:mu_values}, taken from~\cite{CPS} and~\cite{ATAEC}.

  \begin{table}
    \begin{tabular}{|c||c|c|c|} \hline
      type    & $v(\Delta)$ & $\mu$                        & $\alpha$  \lstrut\\\hline\hline
     I$_m$    & $m \ge 2$   & $i(m-i)/m,\; i=1,\ldots,m-1$ & $m/4$     \lstrut\\\hline
     III      & $\ge 3$     & $1/2$                        & $1/2$     \lstrut\\\hline
     IV       & $\ge 4$     & $2/3$                        & $2/3$     \lstrut\\\hline
     I$^*_m$  & $\ge 6+m$   & 1, $(m+4)/4$                 & $(m+4)/4$ \lstrut\\\hline
     IV$^*$   & $\ge 8$     & $4/3$                        & $4/3$     \lstrut\\\hline
     III$^*$  & $\ge 9$     & $3/2$                        & $3/2$     \lstrut\\\hline
    \end{tabular}
    \medskip

    \caption{Nonzero values of and upper bounds~$\alpha$ for~$\mu$ for minimal Weierstrass equations}
    \label{T:mu_values}
  \end{table}

  Let $\vmin$ denote the minimal discriminant of~$E$ over~$\O$.
  In general, we have
  \[
    0 \le \mu(P) \le \alpha + \frac{1}{6}\left(v(\Delta(W)) - \vmin\right)
  \]
  by~\cite{CPS}*{Proposition~8}, and (i)~follows from a straightforward computation.
  This also proves~(ii), because
  \[
    \eps(P) = 4\mu(P) - \mu(2P)\,.
  \]
  By the proof of~\cite{CPS}*{Proposition~8}, a transformation from one integral Weierstrass
  equation to another does not change
  $\mu(P)\bmod \Z$, so (iii)~follows from Table~\ref{T:mu_values}.
\end{proof}

Lemmas~\ref{L:fast algo 2} and~\ref{L:mu_eps_bounds} lead to an algorithm for the computation of
$\mu(P)$.
A pair $(x_1,x_2)$ of Kummer coordinates for $P$ is said to be {\em primitive}, if
$\min\{v(x_1),v(x_2)\} = 0$.
Recall that $\pi$ denotes a uniformizer of~$K$.

\begin{algo}\label{algo1} \strut
\begin{enumerate}[1.]\addtolength{\itemsep}{2mm}
  \item Set $B \colonequals v(\Delta)$.
  \item If $B \le 1$, then return~$0$.
        Otherwise set $m \colonequals \lfloor \log(B^3/3)/\log 4 \rfloor$.
  \item Set $\mu_0 \colonequals 0$. Let $(x_1,x_2)$ be primitive
        Kummer coordinates for~$P$ with $(m+1)B+1$ $v$-adic digits of precision.
  \item For $n \colonequals 0$ to~$m$ do:
        \begin{enumerate}[a.]\addtolength{\itemsep}{1mm}
          \item Compute $(x'_1,x'_2) \colonequals \delta(x_1,x_2)$
                (to $(m+1)B+1$ $v$-adic digits of precision).
          \item Set $\ell \colonequals \min\{v(x'_1),v(x'_2)\}$.
          \item If $\ell = 0$, then return $\mu_0$.
          \item Set $\mu_0 \colonequals \mu_0 + 4^{-n-1} \ell$.
          \item Set $(x_1,x_2) \colonequals \pi^{-\ell} (x'_1,x'_2)$
        \end{enumerate}
  \item Return the unique fraction with denominator at most~$B$
    in the interval $[\mu_0,\,\mu_0 + 1/B^2]$.
\end{enumerate}
\end{algo}
We now show that the algorithm is correct and estimate its running time.

\begin{prop} \label{P:fast algo}
  Algorithm~\ref{algo1} computes~$\mu(P)$. Its running time is
  \[ \ll (\log v(\Delta)) \Mult\bigl((\log v(\Delta)) v(\Delta) (\log \#k)\bigr) \]
  as $v(\Delta) \to \infty$, with an absolute implied constant.
\end{prop}

\begin{proof}
  If $B = v(\Delta) \le 1$, then $\mu = \eps = 0$ by Table~\ref{T:mu_values}.
  Otherwise the loop in step~4 computes the sum in Lemma~\ref{L:fast algo 2}
  (where now $M = B \ge 2$). When $\ell = 0$ in step~4c, then $\eps(2^{n'}P) = 0$
  for all $n' \ge n$, hence the infinite sum defining~$\mu$ is actually
  a finite sum and equals~$\mu_0$. (This step could be left out
  without affecting the correctness or the worst-case complexity of the algorithm.)
  Lemma~\ref{L:mu_eps_bounds} shows that $B$ is an upper bound for~$\eps$
  and that $M=B$ is an upper bound for the denominator of~$\mu$. So the algorithm computes~$\mu(P)$,
  provided the precision of $(m+1)B + 1$ $v$-adic digits is sufficient.
  For this, note that the precision loss at each duplication step is
  given by~$\eps(2^n P)$ and is therefore bounded by~$B$. So after at most
  $m+1$~steps in the loop, the resulting~$(x_1,x_2)$ still has at least one digit
  of precision.

  It remains to estimate the running time. We assume that elements of~$\O$
  are represented as truncated power series in~$\pi$, whose coefficients
  are taken from a complete set of representatives for the residue classes.
  Operations on these coefficients can be performed in time $\ll \Mult(\log \#k)$.
  Then steps b through~e in the loop take negligible time compared to step~a,
  which involves a fixed number of additions and multiplications of elements
  given to a precision of $(m+1)B + 1$ digits, leading to
  a complexity of
  \[ \ll \Mult\bigl(((m+1)B + 1) (\log \#k)\bigr) \]
  operations for each pass through the loop. The total running time is therefore
  \[
    \ll (m+1) \Mult\bigl(((m+1)B + 1) (\log \#k)\bigr)
    \ll (\log v(\Delta)) \Mult\bigl((\log v(\Delta)) v(\Delta) (\log \#k)\bigr)
  \]
  as $v(\Delta) \to \infty$.
\end{proof}

\begin{rk}\label{R:silverman_algo}
  We stress that our algorithm does not require $W$ to be minimal.
  If we know that $W$ is minimal, then Silverman's algorithm~\cite{SilvermanHeights}*{\S5},
  which only involves the computation of the valuations of a bounded number of
  polynomials in the coefficients of~$W$ and the coordinates of~$P$,
  can be used to compute~$\mu(P)$.
\end{rk}


\section{Archimedean local heights}\label{arch}

Let $K$ be an archimedean local field with valuation~$v$.
The following methods have been proposed for the computation of the local
height $\hat{\lambda} = \hat{\lambda}_v$ on an elliptic curve~$E/K$, given by a Weierstrass
equation~\eqref{W_eqn}:

\begin{enumerate}[$\bullet$]
  \item An elegant series approach due to Tate and modified by
    Silverman~\cite{SilvermanHeights}.
  \item A more complicated series approach based on theta functions, see~\cite{Cohen}*{Algorithm~7.5.7};
  \item An algorithm based on the Arithmetic Geometric Mean (AGM) and~2-isogenies
    introduced by Bost and Mestre in an unpublished manuscript~\cite{BMisog}, which
    currently requires $v$ to be real; see also Bradshaw's PhD thesis~\cite{BradshawThesis}.
\end{enumerate}

Tate's series converges linearly.
The theta series has a better rate of convergence and is also faster in practice if the elliptic integrals arising in the
algorithm are computed using the AGM.
If $v$ is real and one is interested in high precision, then the method of Bost and Mestre
is preferable, as it converges quadratically.
We now describe this algorithm and provide a complexity analysis.
Let $v$ be real and let $|{\cdot}|$ denote the usual absolute value on $K= \R$.
We want to compute $\hat{\lambda}(P)$ for a point $P \in E(\R)$; for simplicity,
we only consider the case $2P \ne O$.
Note that the function~$\mu$ considered in~\cite{BMisog} satisfies $\mu = \frac{1}{2}\hat{\lambda}$.

Applying a transformation, we may assume that $E$ is given by a Weierstrass equation
\[
  W \colon y^2 = x(x^2+ux+v)\,,
\]
where $u,v \in \R$. If all points of order~$2$ on~$E$ are real, then we set $E_0 = E$.
Otherwise, consider the isogeny $E \to E_0$ defined by
\begin{equation}\label{isog0}
  (x,y) \mapsto\left(\frac{x^2+ux+v}{x},\,y\frac{x^2-v}{x^2}\right)\,,
\end{equation}
where now $E_0$ has full~2-torsion over $\R$ and is given by the Weierstrass equation
\[
 y^2 = x(x^2-2ux+u^2-4v)\,.
\]
By Proposition~\ref{P:bernardi}, it suffices to compute the local height of the image
of~$P$ on~$E_0$ to obtain~$\hat{\lambda}(P)$.
For the algorithm, we need a Weierstrass equation
\[
 W_0 \colon y^2=x(x+a_0^2)(x+b_0^2)
\]
for~$E_0$, where $0 < b_0<a_0\in\R$.
We may assume that $P$ lies on the connected component $E^0_0(\R)$ of the identity;
if not, we can apply the algorithm to $2P \in E^0_0(\R)$ and compute $\hat{\lambda}(P)$ using
\begin{equation}\label{conn-cpt}
  \hat{\lambda}(2P) = 4\hat{\lambda}(P) - \log|2y(P)|\,.
\end{equation}
We define the AGM sequences $(a_n)$ and $(b_n)$ by
\[
 a_n=\frac{a_{n-1}+b_{n-1}}{2}\,, \qquad b_n=\sqrt{a_{n-1}b_{n-1}}\, ,
\]
and we let $M(a_0,b_0)$ denote their common limit.
For $n\ge 1$ we recursively define an elliptic curve~$E_n$ over the
reals by the Weierstrass equation
\[
 W_n \colon y^2=x(x+a_n^2)(x+b_n^2)\, ,
\]
and we define a 2-isogeny $\varphi_{n-1} \colon E_n\to E_{n-1}$ by
\[
 (x,y)\longmapsto\left(\frac{x(x+b_n^2)}{x+a_n^2},\; y\frac{(x+a_{n-1}a_n)(x+b_{n-1}a_n)}{(x+a_n^2)^2}\right)\, .
\]
Then the sequence of curves $(E_n)_n$ converges to a singular cubic curve~$E_\infty$,
with equation
\[
 W_\infty \colon y^2=x\left(x+M(a_0,b_0)^2\right)^2\,.
\]
Moreover, the sequence of isogenies $(\varphi_n)_n$ converges to the identity map
on~$E_\infty(\R)$.

Now let $\hat{\lambda}_n$ denote the local height on~$E_n(\R)$.
Then Proposition~\ref{P:bernardi} asserts that
\begin{equation}\label{isogrel}
  \hat{\lambda}_{n-1}(\varphi_{n-1}(P_n))=2\hat{\lambda}_n(P_n)-\log(x(P_n)+a_n^2)\, ,
\end{equation}
whenever we have $x(\varphi_{n-1}(P_n))\ne0$.

Bost and Mestre use~\eqref{isogrel} to give a formula for~$\hat{\lambda}(P)$.
Note that $\varphi_{n-1}$ maps $E_{n}(\R)$ onto the connected component~$E^0_{n-1}(\R)$ and
that points on~$E^0_{n-1}(\R)$ always have a unique preimage on~$E^0_n(\R)$ under~$\varphi_{n-1}$.
Setting $P_0=P$, we therefore get a well-defined sequence of preimages
$P_n=(x_n,y_n)\in E^0_n(\R)$, which
converges to a point $P_\infty=(x_\infty,y_\infty)\in E_\infty(\R)$.
Here $x_n$ can be calculated using
\[
  x_n=\frac{1}{2}\left(x_{n-1}-a_{n-1}b_{n-1}+\sqrt{(x_{n-1}+a_{n-1}^2)(x_{n-1}+b_{n-1}^2)}\right).
\]
From~\eqref{isogrel} we deduce
\[
  \hat{\lambda}(P) = \hat{\lambda}_0(P)
     = \log\lim_{n\to\infty}\frac{(x_n+a_n^2)^{2^{n-1}}}{\prod^{n-1}_{m=1}(x_m+a_m^2)^{2^{m-1}}}\,,
\]
or equivalently,
\begin{equation}\label{bm_sum}
  \hat{\lambda}(P) = \log(x_1+a_1^2) +
                      \sum^\infty_{n=1}2^{n}\log\frac{x_{n+1} + a_{n+1}^2}{x_{n} + a_{n}^2}\, .
\end{equation}
Because of the quadratic convergence of the~AGM, these formulas can be used to compute
$\hat{\lambda}(P)$ to an accuracy of~$2^{-d}$ in $\ll \log (d + \log \|W\|)$ steps.
This was already shown by Bradshaw, see~\cite{BradshawThesis}*{\S6.1}.
We give a slightly different error estimate.
Note first that we have
\[
    a_n-b_n \le 2^{1-2^n}(a_0-b_0)\, .
  \]
Because $x_n>0$ and $0<b_0<b_n<a_n$, this implies
\begin{equation}\label{aux_bd}
  \frac{a_n^2-b_n^2}{x_n+a_n^2} \le 2^{1-2^n}(a_0-b_0)\frac{a_n+b_n}{x_n+a_n^2} \le
  2^{2-2^n}\left(\frac{a_0}{b_0}-1\right)\, .
\end{equation}
Now set
\[
  s_n := 1 - \frac{x_{n+1}+a_{n+1}^2}{x_n+a_n^2}
\quad\textrm{and}\quad\vartheta := \frac{a_0}{b_0}-1 + \sqrt{\frac{a_0}{b_0}-1}\, .
\]
Then we have $0<s_n<1$ and $\vartheta \ll \|W\|$. The sequence $s_n$ converges rapidly to~0 for
large $n$, since~\eqref{aux_bd} implies
\begin{align}
  s_n
& \, = \left|\frac{1}{2}\left(\sqrt{\frac{x_n+b_n^2}{x_n+a_n^2}} +
\frac{x_n+\left(\frac{a_n^2+b_n^2}{2}\right)}{x_n+a_n^2}\right)-1\right|\nonumber\\
& \,\le \frac{1}{2}\left|\sqrt{\frac{x_n+b_n^2}{x_n+a_n^2}} -1\right| + \frac{1}{2}\left|\frac{x_n+\left(\frac{a_n^2+b_n^2}{2}\right)}{x_n+a_n^2}-1\right|\nonumber\\
& \,\le \frac{1}{2}\sqrt{\frac{a_n^2-b_n^2}{x_n+a_n^2}}  + \frac{1}{4}\left(\frac{a_n^2-b_n^2}{x_n+a_n^2}\right)\nonumber\\
& \,\le 2^{-2^{n-1}}\sqrt{\frac{a_0}{b_0}-1}+   2^{-2^n}\left(\frac{a_0}{b_0}-1\right) \nonumber\\
& \,\le 2^{-2^{n-1}}\vartheta\label{s_n-bound}\, .
\end{align}
In particular, we have $s_n \le \frac{1}{2}$ for $n \ge \log_2(\log_2\vartheta + 1) + 1$, so that
$|\log(1-s_n)| \le 2s_n$ for such $n$.
We can use this to bound the tail of the series in~\eqref{bm_sum}. Namely, we have
\begin{align*}
  \left|\sum^\infty_{n=N+1}2^{n}\log\frac{x_{n+1} + a_{n+1}^2}{x_{n} + a_{n}^2} \right|
& \,\le \sum^\infty_{n=N+1}2^{n}\left|\log(1-s_n) \right|\\
& \,\le \vartheta\sum^\infty_{n=N+1}2^{1+n-2^{n-1}}\, ,
\end{align*}
if $N \ge \log_2(\log_2\vartheta + 1)$.
For $n \ge 4$, we have $n-2^{n-1} \le -2^{n-2}$, so
\begin{equation}\label{est_n4}
  \left|\sum^\infty_{n=N+1}2^{n}\log\frac{x_{n+1} + a_{n+1}^2}{x_{n} + a_{n}^2} \right|
  \le \vartheta\sum^\infty_{n=N+1}2^{1-2^{n-2}}
  \le  2^{2-2^{N-1}}\vartheta
\end{equation}
follows, provided $N \ge \max\{3, \log_2(\log_2\vartheta + 1) \}$.

Having computed~$\hat{\lambda}(P)$ for $P \in E(\R)$, we get~$\Psi_\infty(P)$ from
\begin{equation}\label{psi-lambda}
  \Psi_\infty(P) = \log \max\{1,|x(P)|\}-\hat{\lambda}(P)\, .
\end{equation}
\begin{prop}\label{P:arch}
  The algorithm above computes $\Psi_\infty(P)$ to $d$~bits of precision in time
\[
  \ll \log(d + \log\|W\|)^2 \Mult(d + \log\|W\|)\, .
\]
\end{prop}

\begin{proof}
  Suppose first that we have already computed $a_0, b_0$ and $x_0$ and that $P$ lies on
  the connected component $E_0^0(\R)$.
  By~\eqref{bm_sum} and~\eqref{est_n4}, we have
  \[
  \left|\hat{\lambda}(P) - \log(x_1+a_1^2) -
                      \sum^N_{n=1}2^{n}\log\frac{x_{n+1} + a_{n+1}^2}{x_{n} +
                      a_{n}^2}\right| \le 2^{-d}\,,
  \]
  for
  \[
    N = \max\left\{3,\, \log_2\left(d+2+\log_2 \vartheta)\right)+1\right\} \ll \log (d + \log \|W\|)\, .
  \]
  For every $n \le N$, we have to apply a fixed number of additions,
  multiplications and square roots to compute $a_{n+1}, b_{n+1}$ and $x_{n+1}$ --- which can
  be done to $d'$~bits of precision in time $\ll \Mult(d')$ --- and we have to compute
  $\log(1-s_n)$.
  Because of precision loss due to the multiplication by $2^n$, we need  to compute
  $\log(1-s_n)$ to  an additional $n$~bits, so we need an initial precision of
  \[
    d  + N \ll d  + \log (d + \log \|W\|)
  \]
  bits. A logarithm can be computed to $d'$~bits of precision in
  time $\ll (\log d')\Mult(d')$ using one of several quadratically converging algorithms
  based on the AGM, see~\cite{Borwein}*{Chapter~7}.
  Therefore, and by~\eqref{s_n-bound}, we can compute $\log(1-s_n)$ to $d+n$~bits of precision
  in time
  \[
    \ll \log(d + \log(d + \log \|W\|))\Mult(d + \log (d + \log \|W\|))\, .
  \]
  The computation of $\log(x_1+a_1^2)$ to $d$~bits of precision takes time
  \[
     \ll \log(d + \log\|W\|)\Mult(d + \log\|W\|) \, .
  \]
  Hence, given $a_0$, $b_0$ and $x_0$ to $d + N$~bits of precision, we can compute
  $\hat{\lambda}(P)$ to $d$~bits of precision in time
  \begin{align*}
    &\ll \log(d + \log\|W\|)\, \times\\ &\qquad{} \left(\Mult(d + \log\|W\|) + \log\bigl(d + \log(d + \log\|W\|)\bigr)\Mult\bigl(d + \log(d +
    \log\|W\|)\bigr)\right)\, .
  \end{align*}
  We can then find $\Psi_\infty(P)$ using~\eqref{psi-lambda} in time $\ll \log(d)
  \Mult(d)$, which is negligible.

  To compute $a_0$, $b_0$ and $x_0$ from a given Weierstrass equation, we need to
  find the roots of at most two polynomials of degree $\le 3$ with real coefficients,
  transform the corresponding Weierstrass equation and find the image of our point $P$
  under these transformations.
  The roots of a polynomial of fixed degree to $d'$~bits of precision can be found in time
  $\ll \Mult(d')$, see~\cite{Borwein}*{Theorem~6.4}; the same holds for the evaluation of
  a polynomial of fixed degree.
  To counter loss of precision, we start with an initial precision of $\ll d + \log
  \|W\| + \log(d + \log\|W\|)$ bits, so we can compute $a_0$, $b_0$ and $x_0$ to $d + N$~bits of precision in time
  \[
    \ll \Mult(d + \log \|W\| + \log(d + \log\|W\|))\, ,
  \]
  which is dominated by the complexity of the remaining parts of the algorithm.
  The logarithmic correction terms coming from~\eqref{conn-cpt} and from Proposition~\ref{P:bernardi} applied
  to the isogeny~\eqref{isog0} and to the change of model needed to find $W_0$ can be
  computed to sufficiently many bits of precision in time $\ll \log(d + \log \|W\|)\Mult(d + \log
  \|W\|)$.
  Hence the result follows.
\end{proof}

\begin{rk}\label{R:series}
  For large $n$, computing $\log(1-s_n)$ using an AGM-based
  algorithm might be less efficient than using a power series such as
  \[
    \log x = 2\sum^{\infty}_{k=0}\frac{1}{2k+1}\left(\frac{x-1}{x+1}\right)^{2k+1}\, .
  \]
  The reason is that by~\eqref{s_n-bound}, the numbers $1-s_n$
  are very close to~1, so only few terms of the power series have to be computed.
\end{rk}


\section{Computing the canonical height of rational points}\label{algo}

We combine the results of Sections~\ref{nonarch} and~\ref{arch} into an efficient
algorithm for computing the canonical height of a point $P$ on an elliptic curve $E$
over a number field, proving Theorem~\ref{T:main}. For simplicity, we take this number
field to be~$\Q$ in the following.
We assume that our curve is given by a Weierstrass equation~\eqref{W_eqn} $W$ with
coefficients in~$\Z$, but we make no minimality assumption.

One usually computes~$\hat{h}(P)$ using the decomposition~\eqref{loc_height_decomp} into
local heights~$\hat{\lambda}_v(P)$. The local height~$\hat{\lambda}_\infty(P)$ can be computed using
the algorithm of Bost-Mestre discussed in Section~\ref{arch} or one of the other
approaches mentioned there.
If the factorization of~$\Delta(W)$ is known, we can use~\cite{SilvermanHeights}*{\S5} to
compute the local heights~$\hat{\lambda}_p(P)$ efficiently.
Alternative, but less efficient algorithms can be found in~\cite{TschoepeZimmer}
and~\cite{Zimmer}.
If we know that $W$ is minimal (for which some factorization is required, see the
introduction), then we can use~\cite{SilvermanLittleFact} to compute
$\sum_p \hat{\lambda}_p(P)$ without factoring~$\Delta(W)$.
Another approach to computing~$\hat{h}(P)$ without factorization is discussed
in~\cite{EverestWard}, but their method does not yield a polynomial-time algorithm.

Our goal is to devise an algorithm for the computation of~$\hat{h}(P)$
that runs in time quasi-linear in $\log \|W\|$, $h(P)$ and the required
precision~$d$, measured in bits after the binary point. We note that $h(P)$ is the logarithm of a
rational number, so it can be computed in time
$\ll \log(h(P) + d) \Mult(h(P) + d)$.
In the previous section, we showed that there is a quasi-linear algorithm for the
computation of $\Psi_\infty(P)$, see Proposition~\ref{P:arch}.

It remains to see how the total contribution
\[{\Psi}^{\text{f}}(P) \colonequals \sum_p \Psi_p(P) = \sum_p \mu_p(P) \log p\]
coming from the local error functions at finite places can be computed efficiently;
here we write~$\mu_p$ for the local height correction function
over~$\Q_p$ as in Definition~\ref{mu_def}.

Fix $P \in E(\Q)$.
We assume that $(x_1,x_2)\in \Z^2$ is a primitive (i.e., $\gcd(x_1,x_2) = 1$)
pair of Kummer coordinates for~$P$.
We set $g_n = \gcd(\delta(x^{(n)}_1,x^{(n)}_2))$ where
$(x^{(n)}_1,x^{(n)}_2) \in \Z^2$ is a primitive pair of Kummer coordinates for~$2^n P$.
Then the definition of $\mu_p$ implies that
\[ {\Psi}^{\text{f}}(P) = \sum_{n=0}^\infty 4^{-n-1} \log g_n \,. \]
See~\cite{FlynnSmart} for a related approach in genus~2.
By Lemma~\ref{L:mu_eps_bounds} we know that each~$g_n$ divides~$\Delta(W)$.
The key observation is that ${\Psi}^{\text{f}}(P)$ is a rational linear combination
of logarithms of positive integers, which can be computed exactly as follows.

\begin{algo}\label{algo2}\strut
\begin{enumerate}[1.]\addtolength{\itemsep}{2mm}
  \item Set $(x_1',x_2') \colonequals \delta(x_1,x_2)$, $g_0 \colonequals \gcd(x'_1,x'_2)$
    and $(x_1,x_2) \colonequals (x'_1/g_0,x'_2/g_0)$.
  \item Set $D \colonequals \gcd(\Delta(W), g_0^\infty)$
        and $B \colonequals \lfloor \log D/\log 2 \rfloor$.
  \item If $B \le 1$, then return $0$.
        Otherwise set $m \colonequals \lfloor \log(B^5/3)/\log 4 \rfloor$.
  \item For $n \colonequals 1$ to $m$ do:
        \begin{enumerate}[a.]\addtolength{\itemsep}{1mm}
          \item Compute $(x'_1,x'_2) \colonequals \delta(x_1,x_2) \bmod D^{m+1} g_0$.
          \item Set $g_n \colonequals \gcd(D,\gcd(x'_1,x'_2))$ and $(x_1,x_2) \colonequals
            (x'_1/g_n,x'_2/g_n)$.
        \end{enumerate}
  \item Use Bernstein's algorithm from~\cite{dcba2} to compute
        a sequence $(q_1, \ldots, q_r)$ of pairwise coprime positive integers
        such that each~$g_n$ (for $n = 0, \ldots, m$) is a product of powers
        of the~$q_i$: $g_n = \prod_{i=1}^r q_i^{e_{i,n}}$.
  \item For $i \colonequals 1$ to $r$ do:
        \begin{enumerate}[a.]\addtolength{\itemsep}{1mm}
          \item Compute $a \colonequals \sum_{n=0}^{m} 4^{-n-1} e_{i,n}$.
          \item Let $\mu_i$ be the simplest fraction between $a$ and $a+1/B^4$.
        \end{enumerate}
  \item Return $\sum_{i=1}^r \mu_i \log q_i$, a formal linear combination of logarithms.
\end{enumerate}
\end{algo}

\begin{prop} \label{P:nofact}
  The preceding algorithm computes ${\Psi}^{\text{\rm f}}(P)$ in time
  \[ \ll (\log\log D) \Mult((\log\log D) (\log D)) + \Mult(h(P)) \,. \]
\end{prop}

\begin{proof}
  We note that if $B \le 1$ in step~3, then either $g_0 = 1$ and $\Psi^{\text{\rm f}}(P) = 0$,
  or else $D \in \{2,3\}$. In the latter case, $g_0$ is a power of $p = 2$ or~$3$
  and $v_p(\Delta(W)) = 1$, which would imply that $\eps_p(P) = 0$, so $g_0 = 1$,
  and we get a contradiction.

  Let $p$ be a prime. If $p \nmid g_0$, then $\eps_p(P) = 0$ and therefore $\mu_p(P) = 0$.
  So we now assume that $p$ divides~$g_0$. We have $v_p(\Delta(W)) = v_p(D) \le B$.
  We see that $p^{(m+1)v_p(D) + 1}$ divides $D^{m+1} g_0$, so computing
  modulo~$D^{m+1} g_0$ will provide sufficient $p$-adic accuracy so that
  $v_p(g_n) = \eps_p(2^n P)$ for all $n \le m$, compare the proof of
  Proposition~\ref{P:fast algo} above.
  (One could replace $D^{m+1} g_0$ by~$D^{m+1-n} g_0$ in step 4a.)
  Since all the~$g_n$ are power products of the~$q_i$, there will be exactly
  one~$i = i(p)$ such that $p \mid q_{i(p)}$; let $b_p = v_p(q_{i(p)})$.
  Then
  \[ \sum_{n=0}^m 4^{-n-1} \eps_p(2^n P)
       = \sum_{n=0}^m 4^{-n-1} v_p(g_n)
       = b_p \sum_{n=0}^m 4^{-n-1} e_{i(p),n}
       = b_p a \,,
  \]
  so
  \[ \mu_p(P) = \sum_{n=0}^\infty 4^{-n-1} \eps_p(2^n P)
              = b_p a + \sum_{n=m+1}^\infty 4^{-n-1} \eps_p(2^n P) \,,
  \]
  where the last sum is in $[0, 1/B^4]$ (this follows from $0 \le \eps_p \le B$,
  see Lemma~\ref{L:mu_eps_bounds},
  and the definition of~$m$, compare the proof of Lemma~\ref{L:fast algo 2}).
  We know that the denominator
  of~$\mu_p(P)$ is at most~$B$ (see Lemma~\ref{L:mu_eps_bounds}), so the
  denominator of~$\mu_p(P)/b_p$ is at most~$B^2$, since $b_p\le v_p(D) \le B$.
  On the other hand, $a \le \mu_p(P)/b_p \le a + 1/(b_p B^4) \le a + 1/B^4$,
  which implies that $\mu_p(P)/b_p$ is the unique fraction in $[a, a + 1/B^4]$
  with denominator at most~$B^2$, so $\mu_p(P)/b_p = \mu_{i(p)}$ by Step~6b.
  Now
  \[ \sum_p \mu_p(P) \log p
       = \sum_p \mu_{i(p)} b_p \log p
       = \sum_{i=1}^r \mu_i \sum_{p \mid q_i} b_p \log p
       = \sum_{i=1}^r \mu_i \log q_i \,.
  \]

  It remains to estimate the running time. The computation of~$\delta(x_1,x_2)$
  can be done in time $\ll \Mult(h(P))$;
  the same is true for the gcd computation and the division in step~1.
  The computations in steps 2 and~3 take negligible time compared to step~4.
  Each pass through the loop in step~4 takes time $\ll \Mult\bigl((m+2) \log D\bigr)$,
  so the total time for the loop is
  $\ll m \Mult(m (\log D)) \ll (\log\log D) \Mult((\log\log D) (\log D))$.
  The algorithm in~\cite{dcba2} computes suitable~$q_i$ for a pair
  $a, b$ of positive integers in time $\ll (\log ab)(\log\log ab)^{2}$.
  We iterate this algorithm,
  applying it first to $g_0$ and~$g_1$, then to each of the resulting $q_i$ and~$g_2$,
  and so on. Note that $g_n\le D$ for all $n$. Because there are always
  $\ll \log D$ terms in the sequence of~$q_i$'s, this leads to a contribution
  of $\ll \log D (\log\log D)^3$ for step~5. This is dominated by the time for the loop.
  The remaining steps take negligible time.
\end{proof}

In practice, the efficiency of this approach can be improved as follows:
\begin{enumerate}[$\bullet$]
  \item We trial factor $\Delta(W)$ up to some bound~$T$ to split off the contributions
        of all sufficiently small primes~$p$. We can then compute the corresponding~$\mu_p$
        using the algorithm of Proposition~\ref{P:fast algo} or the algorithm
        of~\cite{SilvermanHeights}, see Remark~\ref{R:silverman_algo}.
        In step~3, we can then set $B \colonequals \lfloor \log D'/\log T \rfloor$,
        where $D'$ is the unfactored part of~$D$.
        Note that in practice the trial division can take quite a bit more time
        than it saves, in particular when the equation has large coefficients,
        so this modification should be used with care.
  \item We update our list of `building blocks'~$q_i$ after each pass through the
        loop in step~4 using the new~$g_n$; we do the computation modulo suitable
        powers of the~$q_i$ instead of modulo~$D^{m+1} g_0$.
        We can also use separate values of $B$ and~$m$ for each~$q_i$,
        which will usually be smaller than those given above.
  \item In this way, we can integrate steps 4, 5 and~6 into one loop.
  \item We can replace $B^5$ in the definition of~$m$ by~$2B^4$. Then
        $\mu_p(P) \le b_p a + 1/(2B^3)$ and $a \le \mu_p(P)/b_p \le a + 1/(2b_pB^3)$.
        If $\mu_p(P)/b_p = r/s$ with $s \le B b_p$, then we have
        $a \le r/s \le a + 1/(2sB^2) \le a + 1/(2s^2)$.
        There can be at most one fraction~$r/s$ with $s \le B^2$ satisfying this:
        if $r'/s'$ is another such fraction, then
        \[ \frac{1}{ss'} \le \Bigl|\frac{r}{s} - \frac{r'}{s'}\Bigr| \le \frac{1}{2\min\{s,s'\}B^2} \,, \]
        which leads to the contradiction $\max\{s,s'\} \ge 2B^2$.
        We can then find $\mu_i = \mu_p(P)/b_p$ as the first convergent~$r/s$
        of the continued fraction expansion of~$a$ that is $\ge a$ and satisfies
        $r/s \le a + 1/(2sB^2)$.
\end{enumerate}

Combining Proposition~\ref{P:arch} and Proposition~\ref{P:nofact}, we finally obtain an efficient
algorithm for computing the canonical height~$\hat{h}(P)$ of a point $P \in E(\Q)$.

\begin{proof}[Proof of Theorem~\ref{T:main}]
 The first term is the time needed to compute~$h(P)$.
 The second term comes from the complexity bound for the computation of
 $\Psi^f(P)$ (using $\log D \ll \log \|W\|$) from Proposition~\ref{P:nofact}.
 The third term is the bound for the computation of $\Psi^\infty(P)$ given in
 Proposition~\ref{P:arch}.
\end{proof}


\section{Implementation and Examples} \label{examples}

We have implemented our algorithm using the computer algebra system {\sf Magma}~\cite{Magma}.
In the current implementation,
the factorization into coprimes in the algorithm preceding Proposition~\ref{P:nofact}
uses a relatively simple algorithm due to Buchmann
and Lenstra~\cite{BuchmannLenstra}*{Proposition~6.5} instead of the faster algorithm
of~\cite{dcba2} (or of~\cite{dcba}).
In practice, the running time of this part of the algorithm appears to be negligible.

Let us compare our implementation to {\sf Magma's} built-in command {\tt CanonicalHeight} (version~2.21-2).
The latter uses the method of Bost-Mestre for the computation of the archimedean local
height.
For the finite part of the height, a globally minimal Weierstrass model is computed.
The non-archimedean contributions are then computed separately using the algorithm
from~\cite{SilvermanHeights}; the relevant primes are found by factoring
$\gcd(\delta_1(x_1,x_2),\delta_2(x_1,x_2))$, where $(x_1,x_2)$ is a primitive pair of
Kummer coordinates for a point $P$.
The same strategy is currently used in {\tt Pari/GP}. The computer algebra system {\tt Sage}
contains an implementation of, essentially, Silverman's original algorithm
for the computation of canonical heights from~\cite{SilvermanHeights}; in particular, it
factors the discriminant.

\begin{ex}\label{E:ex1}
Consider the family $E_a$ of curves given by the Weierstrass equation
\[ W_a \colon y^2 = x^3 - ax + a \,, \]
where $a$ is an integer, and the non-torsion point $P = (1, 1)$ on $E_a$.
To compute $\hat{h}(P)$, {\sf Magma} needs to find a globally minimal model for $E_a$,
which boils down to deciding whether a sixth power of a prime divides $a$.
Hence, for random integers $a$ of large absolute value, the {\sf Magma} implementation
becomes slow.
For instance, taking $a$ to be {\tiny
5340200419833800017985460942490398389444339691251749039558531515293241873258929634112121245344691478},
which has~100 digits and is of the form $a = 2\cdot 37\cdot a'$ with $a'$ composite,
{\sf Magma}'s built-in {\tt CanonicalHeight} takes about an hour, but our implementation needs
only 0.001~seconds
to compute $\hat{h}(P)$ to~30 decimal digits of precision.
For these computations, and the ones below, we used a Dell Latitude~E7440 Laptop
with 8~GB of memory and an i5-4300U~CPU with two cores having 1.9~GHz each.
For $a$ equal to {\tiny
11564989338479595339888318793988161304389769478402845252925842502529380219520469639630008648580579144420
\\644034811856542472168315806833370153467480796669618513200953623811052728493745808300717019759850}, which has~200 digits and factors as $a = 2\cdot3^2\cdot5^2
\cdot a'$ with $a'$ composite, the computation of $\hat{h}(P)$ using our implementation
takes 0.003~seconds, whereas {\sf Magma} needs about 5~hours and 30~minutes.

Finally, we look at the~500-digit number  $a =\, ${\tiny
  28276805523181086329328141188416415606304708589734\\77817578971661824087775869298113031993537983620824509955240160299513508612337439203295411762778576874861\\6863628083464269023575658346783517541505391502873826466
  503688549496039448522504993529003411479688448361\\01223685296862173154902553901481398879346590153031505842226530360178416613777225501497807415587146715112\\586124106534351729435112961600134931787708117028525772977
  3270941059335530220433045635898507554473398924\\420918799034729911478550230429211184},
which factors as $a = 2^4\cdot 23\cdot71 \cdot a'$ with $a'$ composite.
Our implementation needs 0.009~seconds to compute $\hat{h}(P)$; {\sf Magma}'s {\tt
CanonicalHeight} did not terminate in~6 weeks.
For this $a$, the computation of the canonical height of $50P$, which has naive height
$h(50P) \approx 1437536.77$, took 0.215~seconds, whereas it took {\sf Magma} 2.83~seconds to even
compute $50P$!

For random $a$ having~5000 digits, the computation of $\hat{h}(P)$
to the standard precision of 30~decimal digits usually takes about 0.7~seconds.
Our implementation is usually faster than {\tt CanonicalHeight} if $a$ has at
least~18 decimal digits.
Note that in contrast to our implementation, the {\sf Magma} implementation of the algorithm of Bost-Mestre for
$\hat{\lambda}_\infty$ is written in {\sf C}.
\end{ex}


\end{document}